\documentclass[11pt]{article}
\usepackage{geometry}
\usepackage{amsmath, amsthm, amssymb, mathrsfs}
\usepackage{graphicx, caption, float, multirow, makecell}

\usepackage[T1]{fontenc}

\usepackage{fancyhdr}

\usepackage[hidelinks]{hyperref}
\usepackage[dvipsnames]{xcolor}
\hypersetup{colorlinks=true}

\usepackage[title]{appendix}

\allowdisplaybreaks

\newtheoremstyle{theorem}
  {10pt}
  {10pt}
  {\sl}
  {\parindent}
  {\bf}
  {. }
  { }
  {}
\theoremstyle{theorem}
\newtheorem{theorem}{Theorem}
\newtheorem{corollary}{Corollary}

\begin{document}
\title{\textbf{Distribution-free testing of linear type}}
\author{Weixuan Xia\thanks{School of Data, Mathematical, and Statistical Sciences, University of Central Florida. \newline\indent\hspace{0.5em} Email: \underline{weixuan.xia@ucf.edu}}}
\date{2026}
\maketitle

\begin{abstract}
  We introduce a distribution-free goodness-of-fit test, termed the omega-1 test, which naturally complements the Kolmogorov--Smirnov test and Cram\'{e}r--von Mises test and can be viewed as their (piecewise) linear analog. Defined as an $\mathrm{L}^{1}$-functional of the empirical process, the test statistic improves on balancing sensitivity to localized and diffuse alternatives and gives a robust and interpretable measure of distributional discrepancy, apart from close connections to the Wasserstein 1-distance. For finite samples, we derive a finite-dimensional computational form for the statistic under general conditions, which leads to various explicit formulas for its null distribution. Under mild continuity assumptions, the limiting statistic is distribution-free, with explicit distribution formulas. In composite settings, the statistic is also compatible with the Khmaladze transformation, enabling asymptotically distribution-free testing. The limiting transformed statistic also has an explicit distribution that escapes reliance on intractable compensator processes or purely numerical evaluation. Simulation results indicate rapid convergence of the finite-sample distributions to their limiting counterparts and support the practical applicability of the test. \medskip\\
  \textbf{MSC2020 Classifications:} 62G10; 62G30; 60E05 \medskip\\
  \textbf{Keywords:} Omega-1 test; goodness of fit; Khmaladze transformation; explicit formulas
\end{abstract}

\newcommand{\dd}{{\rm d}}
\newcommand{\pd}{\partial}
\newcommand{\ii}{{\rm i}}
\newcommand{\PP}{\mathbb{P}}
\newcommand{\E}{\mathbb{E}}
\newcommand{\Var}{\mathrm{Var}}
\newcommand{\Ai}{\mathrm{Ai}}
\newcommand{\Gf}{\mathrm{\Gamma}}
\newcommand{\K}{\mathrm{K}}
\newcommand{\F}{\mathrm{F}}
\newcommand{\G}{\mathrm{G}}
\newcommand{\D}{\mathrm{D}}
\newcommand{\erfc}{\mathrm{erfc}}
\newcommand{\sgn}{\mathrm{sgn}}

\medskip

\section{Introduction}\label{sec:1}

The classical Kolmogorov--Smirnov (KS) and Cram\'{e}r--von Mises (CvM) (or omega-2) tests have occupied a central place in nonparametric goodness-of-fit theory. Both tests are based on the empirical distribution function and compare it to some hypothesized distribution, whereas they also possess notable differences in how their discrepancies are aggregated -- the KS statistic takes a supremum norm (maximum deviation), whereas the CvM statistic integrates the squared deviation. Precisely, if $X=\{X_{1},X_{2},\dots,X_{n}\}\subset\mathbb{R}$ is a sample of \text{i.i.d.} observations, whose empirical distribution function is
\begin{equation*}
  F_{n}(x):=\frac{1}{n}\sum^{n}_{j=1}\mathbf{1}(X_{j}\leq x),\quad x\in\mathbb{R},
\end{equation*}
and suppose that $F_{\theta}$ is a given hypothesized distribution function supported on $\mathbb{R}$ with given parametrization $\theta\in\Theta$. Then, in their standard form, the KS and CvM tests are associated with the null hypothesis that the sample $X$ is drawn from the hypothesized distribution, $\mathrm{H}_{0}:F_{n}=F_{\theta}$, and are directly built from the so-called empirical process,
\begin{equation}\label{EP}
  V_{n,\theta}(x):=\sqrt{n}(F_{n}(x)-F_{\theta}(x)),\quad x\in\mathbb{R}.
\end{equation}

Under the additional assumption that $F_{\theta}$ is absolutely continuous, the KS and CvM tests are considered distribution-free under $\mathrm{H}_{0}$, which explains their enduring appeal in both theoretical and applied statistics. The underlying rationale is to construct a test statistic through a functional of $V_{n,\theta}$ as a statistical distance between $F_{n}$ and $F_{\theta}$. In terms of general $\mathrm{L}^{p}$-functionals for $p\in[1,\infty]$, this test statistic can be represented as the powered Stieltjes integral,
\begin{equation}\label{varpi}
  \varpi^{p}_{n}:=\bigg(\int_{\mathbb{R}}|V_{n,\theta}(x)|^{p}\dd F_{\theta}(x)\bigg)^{1/p}.
\end{equation}
Then, the KS and CvM test statistics correspond to $\varpi^{\infty}_{n}$ (as a proper limit) and $\omega^{2}_{n}=(\varpi^{2}_{n})^{2}$, respectively.

Historically, the importance of the KS and CvM tests comes from providing fully nonparametric, omnibus tests that are sensitive to a wide range of alternatives (such as $\chi^{2}$ tests). The KS test, in particular, is widely implemented across statistical software and used in one- and two-sample settings (\text{cf.} \cite{BZ}, \cite{D}, \cite{MTW}, \cite{SL}), while CvM-type criteria also appear in minimum-distance estimation and related inference procedures (\text{cf.} \cite{A}, \cite{CF}, \cite{D}, \cite{HK2}, \cite{OH}). Over time, both statistics have been extended in numerous directions, including to discrete data, censored data, rank data, and multivariate settings where direct generalization is nontrivial (e.g., via projections or space-filling curves) -- see, e.g., \cite{AE}, \cite{BLFGG}, \cite{CDS}, \cite{JPZ}, \cite{OP-G}, \cite{S1}.

In modern applications, these tests continue to play a significant role beyond classical parametric model checking. This include signal processing and communications, where they appear in empirical likelihood and detection problems (\cite{KLT}, \cite{L}), and in high-dimensional and functional data analysis, where CvM-type statistics are embedded into projection-based or U-statistics to mitigate the curse of dimensionality (\cite{C-AG-PF-B}, \cite{G-PG-MF-B}). In addition, in machine learning, goodness-of-fit ideas are critical for model criticism, generative model evaluation, and distributional two-sample testing; see, e.g., \cite{PCGT}, \cite{SL-MD}, \cite{TT}, \cite{WST} -- while KS-type distances are sometimes used as simple diagnostics, CvM-type integrated discrepancies are closely related to kernel-based methods and integral probability metrics. More recent work also explores smoothed or kernelized versions of the statistics to improve finite-sample behavior and boundary effects, as done in \cite{FM}.

Their versatility notwithstanding, the KS and CvM tests reflect two relatively extreme scenarios, based on the $\mathrm{L}^{\infty}$- and $\mathrm{L}^{2}$-functionals for deviation aggregation, respectively, and this naturally motivates considering an intermediate linear $\mathrm{L}^{1}$-functional, i.e., with $p=1$ in (\ref{varpi}),
\begin{equation*}
  \varpi^{1}_{n}:=\int_{\mathbb{R}}|V_{n,\theta}(x)|\dd F_{\theta}(x).
\end{equation*}
Such a choice has several conceptual and practical advantages. Compared to the KS statistic ($\varpi^{\infty}_{n}$), it reduces the dominance of a single worst-case deviation, while compared to the CvM statistic ($(\varpi^{2}_{n})^{2}$), it also avoids squaring, which can overweight large deviations and underweight moderate but widespread discrepancies. In this sense, $\varpi^{1}_{n}$ is expected to improve robustness to localized irregularities or noise in data by detecting both global and moderately localized deviations and, in so doing, provide a more balanced sensitivity profile.

Besides, from a theoretical perspective, replacing the $\mathrm{L}^{2}$-functional with the $\mathrm{L}^{1}$-functional results in a shift of the asymptotic object under $\mathrm{H}_{0}$ from the square integral of a Gaussian process (or the Brownian bridge in the case of $\varpi^{2}_{n}$ as $n\to\infty$) to its absolute-value integral, which typically exhibits dissimilar tail behaviors -- compare \cite{T1}, \cite{T2}, \cite{X1}, and \cite{XZ} about the asymmetric leptokurtic feature, which can have implications for statistical power against certain alternatives (e.g., sparse \text{vs.} diffuse deviations). From a computational viewpoint, $\mathrm{L}^{1}$-type discrepancies are closely related to Wasserstein distances and optimal transport, which have recently become central in generative modeling and distribution comparison (e.g., see \cite{BFT}, \cite{DBGM}, \cite{PZ}). In this regard, ``linearizing'' the CvM functional also provides a meaningful link between classical goodness-of-fit tests with modern distributional metrics used in machine learning.

With the above considerations, it is natural to refer to the $\mathrm{L}^{p}$-functional-based test as an ``omega-$p$'' test, with test statistic
\begin{equation}\label{Op}
  \omega^{p}_{n}=(\varpi^{p}_{n})^{p},
\end{equation}
and in particular, $\omega^{1}_{n}=\varpi^{1}_{n}$. In Section \ref{sec:2} of this paper, we analyze key properties of this statistic, including the derivation of finite-dimensional representations in the general case of (\ref{Op}), along with various explicit distribution formulas for $p=1$ permitting efficient numerical computation.



\smallskip

Noteworthily, the classical KS and CvM tests are fundamentally reliant on a fully specified null hypothesis $\mathrm{H}_{0}$ under which a distribution-free limit can be guaranteed. However, in practice, one oftentimes faces situations where additional parametric hypotheses about $\theta$ must be verified  (as $\hat{\theta}_{n}$) from the sample $X$ along with $\mathrm{H}_{0}$, with which the latter becomes a partially specified, composite hypothesis. It is well-known that in such situations, the corresponding test statistics are no longer distribution-free, whose null distributions have general dependence on $F_{\hat{\theta}_{n}}$. A powerful tool to resolve this obstacle is the so-called Khmaladze transformation, originally introduced in \cite{K1} and \cite{K2}. Formally, it is constructed as a martingale (or innovation) transformation of the empirical process $V_{n,\theta}$.

More specifically, given $X$, let $L(\theta;X)$ denote the likelihood function associated with $F_{\theta}$, with $|\theta|=m$ parameters, and let
\begin{equation*}
  s(\theta;x):=\nabla_{\theta}\log f_{\theta}(x),\quad \mathcal{I}_{\theta}:=\mathrm{Hess}_{\theta}\log f_{\theta}(x)
\end{equation*}
denote, respectively, the score function and the (incomplete) Fisher information matrix. Then, the transformed empirical process is defined by
\begin{equation}\label{TEP}
  \widetilde{V}_{n,\theta}(x):=V_{n,\theta}(x)-\mathcal{K}_{n,\theta}(x),\quad x\in\mathbb{R},
\end{equation}
with the compensator
\begin{equation}\label{Comp}
  \mathcal{K}_{n,\theta}(x):=\int^{x}_{-\infty}(1,s^{\intercal}_{\theta}(y))\mathcal{I}^{-1}_{\theta}(y) \int^{\infty}_{y}(1,s^{\intercal}_{\theta}(z))^{\intercal}\dd V_{n,\theta}(z)\dd F_{\theta}(y).
\end{equation}
In the above definition, it is clear that $\widetilde{V}_{n,\theta}$ exhibits general dependence on $F_{\theta}$, and so do test statistics built therefrom, including supremum or quadratic functionals. For more general $\mathrm{L}^{p}$-functionals, we shall denote these (distribution-dependent) statistics by
\begin{equation}\label{Wp}
  w^{p}_{n,\theta}=\bigg(\int_{\mathbb{R}}|\widetilde{V}_{n,\theta}(x)|^{p}\dd F_{\theta}(x)\bigg)^{1/p},
\end{equation}
in companion with (\ref{varpi}).

The main significance of the Khmaladze transformation lies in the fact as $n\to\infty$, the transformed process $\widetilde{V}_{n,\theta}$ regains its distribution-free limit, enabling valid and robust testing under the composite null hypothesis ($\mathrm{H}_{0}$) without resorting to resampling or tabulating model-specific critical values. Because of this feature, the Khmaladze transformation has been particularly influential in econometrics, survival analysis, and time-series specification testing, where parameter estimation is unavoidable; see, e.g., \cite{HK1}, \cite{K3}, \cite{ZA}.

Section \ref{sec:3} this paper aims to shed some light on the limit distribution of (\ref{Wp}) in the particular case $p=1$. As the Khmaladze transformation is not tailored to any particular hypothesis test but rather serves as a method for modifying the empirical process, our goal is not to derive explicit finite-sample distributions for specific underlying distributions $F_{\theta}$; instead, it is to highlight the desirable property that applying the Khmaladze transformation in conjunction with the newly proposed $\mathrm{L}^{1}$-functional-based test remains compatible and is expected to produce meaningful testing outcomes.

\medskip

\section{Constructing omega-1 test statistics}\label{sec:2}

\subsection{Finite-sample representations}\label{sec:2.1}

Starting with the aforementioned \text{i.i.d.} sample $X=\{X_{1},X_{2},\dots,X_{n}\}\subset\mathbb{R}$, under the null hypothesis $\mathrm{H}_{0}$, the transformed sample $F_{\theta}(X)=\{F_{\theta}(X_{1}),F_{\theta}(X_{2}),\dots,F_{\theta}(X_{n})\}$ contains \text{i.i.d.} standard uniformly distributed random variables, denoted as $F_{\theta}(X_{j})=U_{j}$, $j=1,2,\dots,n$. Assuming that $F_{\theta}$ is absolutely continuous with fixed $\theta$, we can define the corresponding uniform empirical process
\begin{equation}\label{G}
  G_{n}(t):=\sqrt{n}(\breve{F}_{n}(t)-t),\quad t\in[0,1],
\end{equation}
where
\begin{equation}\label{TECDF}
  \breve{F}_{n}(t):=\frac{1}{n}\sum^{n}_{j=1}\mathbf{1}(U_{j}\leq t)
\end{equation}
is the transformed empirical distribution function, and which does not depend on $\theta$ and allows to rewrite (\ref{varpi}) into the distribution-free form via a change of the variables, namely
\begin{equation}\label{varpia}
  \varpi^{p}_{n}=\bigg(\int^{1}_{0}|G_{n}(t)|^{p}\dd t\bigg)^{1/p}.
\end{equation}

Notably, for $p=1$, the statistic $\omega^{1}_{n}=\varpi^{1}_{n}$ coincides with the Wasserstein 1-distance between the empirical distribution given by (\ref{TECDF}) and the standard uniform distribution. In this case, it is possible to carry out the integral inside (\ref{varpia}) to obtain a finite-dimensional computational form, whose structural complexity is comparable to what is known for the CvM (or $\omega^{2}_{n}$) case, with $p=2$; see, e.g., \cite[\text{Eq.} (1.1)]{CF}. For this purpose, let $U_{(j)}$ denote the $j$th order statistic of $U=\{U_{1},\dots,U_{n}\}$. We then have the following theorem.

\begin{theorem}\label{thm:1}
Under $\mathrm{H}_{0}$, $\omega^{1}_{n}$ admits the computational form
\begin{equation}\label{CF1}
  \omega^{1}_{n}=\varpi^{1}_{n}=\sqrt{n}\sum^{n}_{j=1}
  \left(\begin{cases}
    \displaystyle \frac{j^{2}-nU_{(j)}}{n^{2}},&\quad \displaystyle U_{(j)}\leq\frac{j-1}{n} \\
    \displaystyle \bigg(U_{(j)}-\frac{2j-1}{2n}\bigg)^{2}+\frac{4j(j-1)+3}{4n^{2}},&\quad \displaystyle\frac{j-1}{n}<U_{(j)}\leq\frac{j}{n} \\
    \displaystyle \frac{nU_{(j)}+(j-1)^{2}}{n^{2}},&\quad \displaystyle U_{(j)}>\frac{j}{n}
  \end{cases}\right)
  -\frac{2n^{2}+1}{6\sqrt{n}}.
\end{equation}
\end{theorem}


\begin{proof}
Define the empirical quantile function as the right-inverse of the (transformed) empirical distribution function in (\ref{TECDF}),
\begin{equation*}
  \breve{F}^{-1}_{n}(q):=\inf\{t\in[0,1]:\;\breve{F}_{n}(t)\geq q\},\quad q\in[0,1].
\end{equation*}
It is clear that for every $1\leq j\leq n$, $\breve{F}^{-1}_{n}(q)=U_{(j)}$ for $q\in((j-1)/n,j/n]$.

On the other hand, if $p=1$, by a well-known property of the Wasserstein 1-distance (see, e.g., \cite[\text{Eq.} (1.2)]{DBGM}), we have
\begin{equation*}
  \omega^{1}_{n}=\varpi^{1}_{n}=\sqrt{n}\int^{1}_{0}|\breve{F}_{n}(u)-u|\dd u=\sqrt{n}\int^{1}_{0}|u-\breve{F}^{-1}_{n}(u)|\dd u.
\end{equation*}
This then allows to rewrite (\ref{varpia}) as a sum of integrals over $n$ disjoint intervals, i.e.,
\begin{equation*}
  \varpi^{1}_{n}=\sqrt{n}\int^{1}_{0}|u-\breve{F}^{-1}_{n}(u)|\dd u=\sqrt{n}\sum^{n}_{j=1}\int^{j/n}_{(j-1)/n}|u-U_{(j)}|\dd u.
\end{equation*}
After evaluating these integrals, we obtain
\begin{equation}\label{CF1c}
  \varpi^{1}_{n}=\sqrt{n}\sum^{n}_{j=1}
  \begin{cases}
    \displaystyle\frac{(j-nU_{(j)})^{2}-(j-1-nU_{(j)})^{2}}{2n^{2}},&\quad \displaystyle U_{(j)}\leq\frac{j-1}{n} \\
    \displaystyle\frac{(j-nU_{(j)})^{2}+(nU_{(j)}-j+1)^{2}}{2n^{2}},&\quad \displaystyle \frac{j-1}{n}<U_{(j)}\leq\frac{j}{n} \\
    \displaystyle\frac{(nU_{(j)}-j+1)^{2}-(nU_{(j)}-j)^{2}}{2n^{2}},&\quad \displaystyle U_{(j)}>\frac{j-1}{n}.
  \end{cases}
\end{equation}
which is equivalent to (\ref{CF1}) upon taking out the common factors independent of $U_{(j)}$'s and summing them, namely $\sum^{n}_{j=1}(j^{2}-j+1/2)/n^{2}=(2n^{2}+1)/(6n)$.
\end{proof}


About the proof of Theorem \ref{thm:1}, it is possible to compute the integral in (\ref{varpia}) using (\ref{G}) and (\ref{TECDF}) directly, instead of invoking the functional inversion, which will then lead to an alternative form involving two summations of $n$ terms, hence less efficient, which can be found in Appendix \ref{A}. Theorem \ref{thm:1} also justifies the interpretation of $\omega^{1}_{n}:S_{n}\mapsto\mathbb{R}$ as a function of the order statistics vector $U_{(\cdot)}:=[U_{(1)},U_{(2)},\dots,U_{(n)}]^{\intercal}$, where $S_{n}:=\{y\in[0,1]^{n}:y_{1}\leq y_{2}\leq\cdots\leq y_{n}\}$ is the $n$-dimensional unit simplex.

Moreover, tight bounds are readily available by taking advantage of the linearity of the integrand in (\ref{varpia}) with $p=1$. Indeed, for $n=1$, one has the elementary problem to find a point on the hypotenuse of an isosceles right triangle such that the two sub-triangles obtained by respectively drawing (two) perpendicular line segments from this point towards the two legs has minimal or maximal total area. As a consequence, $\varpi^{1}_{1}$ attains its maximum of $1/2$ when $U_{1}\in\{0,1\}$ and its minimum of $1/4$ when $U_{1}=1/2$, and an induction argument from here implies that
\begin{equation*}
  \frac{1}{4\sqrt{n}}\leq\varpi^{1}_{n}\leq\frac{\sqrt{n}}{2},\quad n\geq1,
\end{equation*}
where the upper bound is attained when either $U_{(1)}=1$ or $U_{(n)}=0$ and the lower bound when $U_{(j)}=(2j-1)/(2n)$ for all $1\leq j\leq n$; both bounds are attained with zero probability.

\medskip

\subsection{Limit null distribution}\label{sec:2.2}

As the sample size increases, the convergence of the omega-$p$ statistic follows from classical empirical process theory. Indeed, Donsker's theorem states that under $\mathrm{H}_{0}$, the uniform empirical process $G_{n}$ converges weakly in $\mathrm{L}^{\infty}([0,1])$ to a standard Brownian bridge $B\equiv(B(t))_{t\geq0}$. Since the $\mathrm{L}^{p}$-functional for $p\in[1,\infty]$ is continuous on $\mathrm{L}^{\infty}([0,1])$, it follows that $\varpi^{p}_{n}$ converges in distribution to the $\mathrm{L}^{p}$-norm of $B$, namely,
\begin{equation*}
  \varpi^{p}_{n}\overset{\rm d}{\to}\varpi^{p}_{\infty},
\end{equation*}
with the limit statistic
\begin{equation*}
  \varpi^{p}_{\infty}:=\bigg(\int^{1}_{0}|B(t)|^{p}\dd t\bigg)^{1/p}=\bigg(\int^{1}_{0}|W(t)-tW(1)|^{p}\dd t\bigg)^{1/p},
\end{equation*}
where $W\equiv(W(t))_{t\geq0}$ is the corresponding standard Brownian motion.

Fortunately, in the case $p=1$, the distribution of $\varpi^{1}_{\infty}=\omega^{1}_{\infty}$ still has an explicit form, with a general formula for its CDF available in \cite[\text{Eq.} 1.8.7.(1)]{BS} expressed in terms of the first Airy function $\Ai\equiv\Ai(\cdot)$. This result is parallel to the cases $p=\infty$ and $p=2$, respectively, for which $\varpi^{\infty}_{\infty}$ and $(\varpi^{2}_{\infty})^{2}$ are known to have the Kolmogorov distribution and the so-called Cram\'{e}r--von Mises distribution; see, e.g., \cite[\text{Sect.} 3]{MTW} and \cite[\text{Thm.} 2]{T2}. The CDFs of the latter two limit statistics are provided below for comparison.
\begin{align}\label{LCDF.i2}
  F_{\varpi^{\infty}_{\infty}}(x)&=\frac{\sqrt{2\pi}}{x}\sum^{\infty}_{k=1}e^{-\pi^{2}(2k-1)^{2}/(8x^{2})}, \nonumber\\
  F_{\varpi^{2}_{\infty}}(x)&=\frac{2}{\sqrt{x}}\sum^{\infty}_{k=0}\frac{(-1)^{k}e^{-(k+1/4)^{2}/x^{2}}}{k!\Gf(1/2-k)} \D_{-1/2}\bigg(\frac{4k+1}{2x}\bigg),\quad x>0,
\end{align}
where $\Gf\equiv\Gf(\cdot)$ and $\D\equiv\D(\cdot)$ denote the usual gamma function and the parabolic cylinder function, respectively. In Theorem \ref{thm:2}, we present an alternative explicit formula for the distribution function of $\varpi^{1}_{\infty}$ that is structurally simpler than the one found in \cite[\text{Eq.} 1.8.7.(1)]{BS}. Although this expression could be derived directly from the formula \text{ibid.} using various known functional relations within the Bessel function family, we provide a self-contained proof for completeness. The derivation here builds on a known Laplace transform result and exploits simplifications arising from the more general Meijer G function, which is also used in Section \ref{sec:3}. For the corresponding PDFs of these statistics presented in similar forms, see also \cite[\text{Eq.} (25)]{X2} and \cite[\text{Thm.} 1]{T2}.

\begin{theorem}\label{thm:2}
Let $\{\alpha'_{k}:\;k\in\mathbb{Z}_{++}\}$ be the zeros of the derivative of $\Ai$, arranged such that $\alpha'_{k}>\alpha'_{k+1}$ for all $k$, and let $\K\equiv\K_{\cdot}(\cdot)$ denote the modified Bessel function of the second kind. Then, the distribution function of $\varpi^{1}_{\infty}$ can be written as
\begin{equation}\label{LCDF.1}
  F_{\varpi^{1}_{\infty}}(x)=\frac{1}{\sqrt{6\pi}x}\sum^{\infty}_{k=1}e^{(\alpha'_{k})^{3}/(27x^{2})} \K_{1/3}\bigg(\frac{(-\alpha'_{k})^{3}}{27x^{2}}\bigg),\quad x>0.
\end{equation}
\end{theorem}

\begin{proof}
Starting with the Laplace transform of the absolute-value integral $\int^{1}_{0}|W(t)|\dd t$ joint with the terminal value of $W_{1}$ given in \cite[\text{Eq.} 1.8.7.(1)]{BS}, we have
\begin{equation*}
  \E\big[e^{-u\int^{1}_{0}W(t)\dd t};W(1)\in\dd
  z\big]=\sum^{\infty}_{k=1}\frac{u^{1/3}e^{(u^{2}/2)^{1/3}\alpha'_{k}}\Ai(\alpha'_{k}+(2u)^{1/3}|z|)}{2^{2/3}(-\alpha'_{k})\Ai(\alpha'_{k})}\dd
  z,\quad\Re u>0,\;z\in\mathbb{R}.
\end{equation*}
Since $W(1)$ has a standard normal distribution, this implies that
\begin{equation}\label{LTvarpi1}
  \E[e^{-u\varpi^{1}_{\infty}}]=\E\big[e^{-u\int^{1}_{0}|W(t)-tW(1)|\dd t}\big] =\sqrt{2\pi}\sum^{\infty}_{k=1}\frac{u^{1/3}e^{(u^{2}/2)^{1/3}\alpha'_{k}}}{2^{2/3}(-\alpha'_{k})}, \quad\Re u>0.
\end{equation}
Noting that by (\cite[\text{Eqs.} 10.4.95 and 10.4.105]{AS}), as $k\to\infty$,
\begin{equation*}
  \alpha'_{k}=-\bigg(\frac{3\pi(4k-3)}{8}\bigg)^{2/3}(1+O(k^{-2})),
\end{equation*}
the series in (\ref{LTvarpi1}) converges absolutely, and termwise inversion can be performed. The integration property of Laplace transform entails that we invert $\E[e^{-u\varpi^{1}_{\infty}}]/u$, which by the inversion formula gives that
\begin{equation}\label{Fvarpi1}
  F_{\varpi^{1}_{\infty}}(x)
  =\frac{\sqrt{2\pi}}{2^{2/3}}\sum^{\infty}_{k=1}\frac{1}{-\alpha'_{k}}\frac{1}{2\pi\ii}\int^{c+\ii\infty}_{c-\ii\infty} \frac{e^{ux-(-\alpha'_{k}/2^{1/3})u^{2/3}}}{u^{2/3}}\dd u,
\end{equation}
for arbitrary $c>0$, which amounts to inverting $e^{-u^{2/3}}/u^{2/3}$ in $u$. To this end, we may consult \cite[\text{Eq.} 2.2.1.19]{PCGT} to obtain a general expression in terms of the Meijer G function,
\begin{align}\label{ILT.G}
  \frac{1}{2\pi\ii}\int^{c+\ii\infty}_{c-\ii\infty}\frac{e^{ux-u^{2/3}}}{u^{2/3}}\dd u
  &=\sqrt{\frac{3}{\pi}}\frac{1}{x}\bigg(\frac{x}{2}\bigg)^{2/3}\G^{3,0}_{2,3}\bigg(
  \begin{array}{ccccc}
     &  &  & 1/3 & 5/6 \\
    0 & 1/3 & 2/3 &  &
  \end{array}
  \bigg|\frac{4}{27x^{2}}\bigg) \nonumber\\
  &=\sqrt{\frac{3}{\pi}}\frac{1}{x}\bigg(\frac{x}{2}\bigg)^{2/3}\G^{2,0}_{1,2}\bigg(
  \begin{array}{cccc}
     &  &  & 5/6 \\
    0 & 2/3 &  &
  \end{array}
  \bigg|\frac{4}{27x^{2}}\bigg) \nonumber\\
  &=\frac{e^{-2/(27x^{2})}}{\sqrt{3}\pi}\K_{1/3}\bigg(\frac{2}{27x^{2}}\bigg),\quad x>0,
\end{align}
where the second equality uses the basic identity in \cite[\text{Eq.} 5.3.1.(7)]{EMOT} and the third equality uses \cite[\text{Eqs.} 5.6.(5) \& 6.9.1.(14)]{EMOT} for reduction into the Bessel $\K$ function under the particular dimensions. With the scaling property of Laplace transform, the desired formula (\ref{LCDF.1}) follows from plugging (\ref{ILT.G}) into (\ref{Fvarpi1}) and simplifying.
\end{proof}

In light of the absolute convergence of the series in (\ref{LTvarpi1}), by using the differentiation formula in \cite[\text{Eq.} 10.2.22]{AS} for the Bessel $\K$ function, the termwise differentiation of (\ref{LCDF.1}) readily yields a formula\footnote{To our knowledge, this formula (even in modified form) has not appeared in prior work.} for the density function of $\varpi^{1}_{\infty}$,
\begin{equation}\label{LPDF.1}
  f_{\varpi^{1}_{\infty}}(x)=\frac{1}{27\sqrt{6\pi}x^{4}}\sum^{\infty}_{k=1}e^{(\alpha'_{k})^{3}/(27x^{2})} \bigg((2(-\alpha'_{k})^{3}-9x^{2})\K_{1/3}\bigg(\frac{(-\alpha'_{k})^{3}}{27x^{2}}\bigg) +2(-\alpha'_{k})^{3}\K_{2/3}\bigg(\frac{(-\alpha'_{k})^{3}}{27x^{2}}\bigg)\bigg),
\end{equation}
for $x>0$, which is clearly infinitely smooth as so is the Bessel $\K$ function.

Both formulae (\ref{LCDF.1}) and (\ref{LPDF.1}) are amenable to numerical implementation owing to the rapid convergence of the series involving terms decaying at a quadratic exponential rate for small arguments, which we explain in detail in Corollary \ref{cor:1}.

\begin{corollary}\label{cor:1}
Based on (\ref{LCDF.1}) and (\ref{LPDF.1}), let $F_{\varpi^{1}_{\infty}}(x)=\sum^{\infty}_{k=1}H_{k}(x)$ and $f_{\varpi^{1}_{\infty}}(x)=\sum^{\infty}_{k=1}h_{k}(x)$. Then, for any $x>0$, there exists a positive integer $K\equiv K(x)$ such that
\begin{align*}
  \bigg|\sum^{\infty}_{k=K+1}H_{k}(x)\bigg|&<\frac{|H_{K}(x)H_{K+1}(x)|}{|H_{K}(x)|-|H_{K+1}(x)|}, \\
  \bigg|\sum^{\infty}_{k=K+1}h_{k}(x)\bigg|&<\frac{|h_{K}(x)h_{K+1}(x)|}{|h_{K}(x)|-|h_{K+1}(x)|}.
\end{align*}
\end{corollary}

\begin{proof}
For the Bessel $\K$ function, we have the asymptotic (\cite[\text{Eq.} 9.7.2]{AS})
\begin{equation}\label{Kas}
  \K_{1/3}(z)=\sqrt{\frac{\pi}{2}}e^{-z}(z^{-1/2}+O(z^{-3/2})),\quad\text{as }z\to\infty.
\end{equation}
This along with (\ref{LCDF.1}) implies that for fixed $x>0$,
\begin{equation*}
  H_{k}(x):=e^{(\alpha'_{k})^{3}/(27x^{2})}\K_{1/3}\bigg(\frac{(-\alpha'_{k})^{3}}{27x^{2}}\bigg)=\sqrt{\frac{\pi}{2}}e^{2(\alpha'_{k})^{3}/(27x^{2})} \bigg(\sqrt{\frac{27x^{2}}{(-\alpha'_{k})^{3}}}+O(k^{-1})\bigg),\quad\text{as }k\to\infty,
\end{equation*}
which shows that $H_{k}(x)=o(e^{-k})$ as $k\to\infty$, and there exists a constant $K\in\mathbb{Z}_{++}$ such that $|H_{k+1}(x)|/|H_{k}(x)|<1$ for all $k\geq K$. Thus, we have
\begin{equation*}
  \bigg|\sum^{\infty}_{k=K+1}H_{k}(x)\bigg|\leq\sum^{\infty}_{k=K+1}|H_{k}(x)|<\frac{|H_{K}(x)H_{K+1}(x)|}{|H_{K}(x)|-|H_{K+1}(x)|}.
\end{equation*}
The same arguments apply to the function $h_{k}(x)$, $x>0$.
\end{proof}

Based on Corollary \ref{cor:1}, truncating the infinite series in (\ref{LCDF.1}) at $k=3$ guarantees an absolute error less than $10^{-6}$ for all $x\in[0,1]$, i.e., $|F_{\varpi^{1}_{\infty}}(x)-\sum^{3}_{k=1}H_{k}(x)|\leq10^{-6}$, which lays the foundation for efficient computation. Additionally, the unit interval is sufficient for possible tabulation because $\PP\{\varpi^{1}_{\infty}>1\}\leq10^{-3}$, while for arbitrarily large $x$, the following asymptotic formula due to \cite{T1} may be used:
\begin{equation*}
  \PP\{\varpi^{1}_{\infty}>x\}=\frac{1}{\sqrt{6\pi}x}e^{-6x^{2}}+O(x^{-2}),\quad\text{as }x\to\infty.
\end{equation*}

In Figure \ref{fig:1}, we compare the limit distributions of the KS statistic, the (square root) CvM statistic, and the omega-1 statistic in consideration, corresponding to $\varpi^{p}_{\infty}$ with $p=\infty$, $p=2$, and $p=1$, respectively, by implementing (\ref{LCDF.i2}) and (\ref{LCDF.1}) for $x\in[0,2]$, all with the infinite series truncated at $k=4$ (with absolute error less than $10^{-5}$). It is clear that $\varpi^{\infty}_{\infty}$ dominates $\varpi^{2}_{\infty}$, which further dominates $\varpi^{1}_{\infty}$ in stochastic order; comparably, the $\mathrm{L}^{1}$-functional aggregates moderate deviations, whereas the $\mathrm{L}^{2}$-functional tends to amplify extrema, causing crossings of their limit distributions.

\begin{figure}[H]
  \centering
  \includegraphics[scale=0.5]{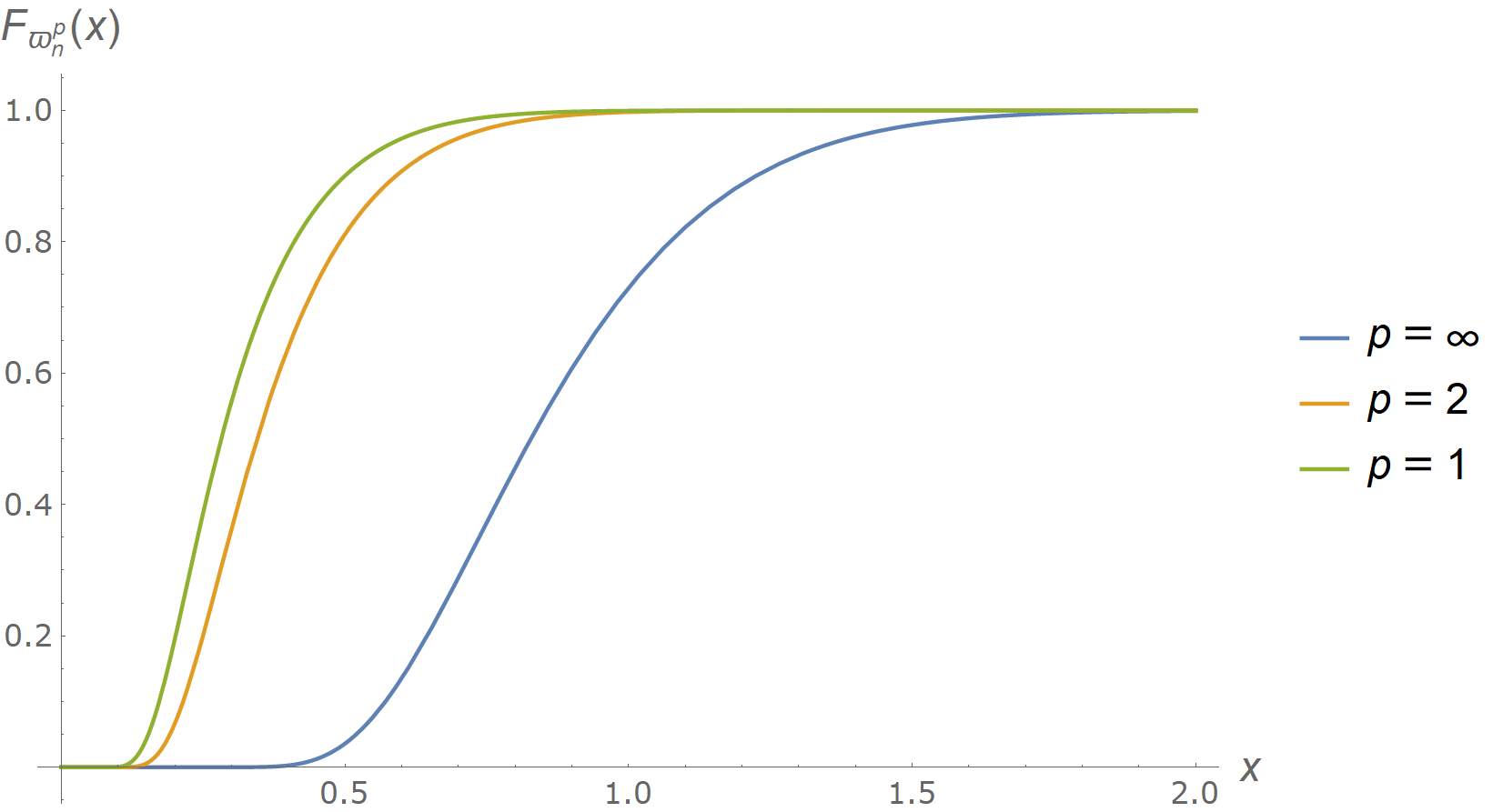}
  \caption{Known limit distributions of $\varpi^{p}_{n}$}
  \label{fig:1}
\end{figure}

Regarding finite-sample convergence, we can obtain a crude rate for the distribution function by combining $\mathrm{L}^{1}$-convergence of the uniform empirical process in (\ref{G}) with the boundedness of the limiting density (\ref{LPDF.1}). Instead of seeking sharper bounds via functional delta methods or higher-order expansions, we shall turn to numerical illustrations in Subsection \ref{sec:2.3}. The results indicate that convergence to the limit distribution is remarkably rapid in practice, which can be largely ascribed to the presence of a seemingly small constant, along with the smoothing effect of the integral in (\ref{varpia}). In particular, the asymptotic approximation appears to be already highly accurate for a sample size as small as $n\geq10$, and consequently, unlike the case of the (KS) $\varpi^{\infty}_{n}$ statistic, where finite-sample corrections are often needed to compensate for slower convergence (see, e.g., \cite{S3}, \cite{MTW}), such refinements do not seem to be critical for $\varpi^{1}_{n}$. The crude rate is presented in Corollary \ref{cor:2}.\footnote{The $\mathrm{L}^{1}$-functional here is smoother than the $\mathrm{L}^{\infty}$-functional underlying the KS statistic, which contributes to improved finite-sample behavior, while it remains comparable in structure to the $\mathrm{L}^{2}$-functional defining the omega-1 statistic.}

\begin{corollary}\label{cor:2}
There exists a constant $C>0$ such that\footnote{It is known that the Esseen constant $C_{0}\leq0.4748$ (see, e.g., \cite{S3}), while numerical estimates based on (\ref{LPDF.1}) show that $C_{1}\leq4$, and so by (\ref{lb}), the constant $C$ in (\ref{CDF1.conv}) is at most $2\sqrt{2C_{0}C_{1}}\approx3.8979$.}
\begin{equation}\label{CDF1.conv}
  \sup_{x>0}|F_{\varpi^{1}_{n}}(x)-F_{\varpi^{1}_{\infty}}(x)|\leq Cn^{-1/4}.
\end{equation}
\end{corollary}

\begin{proof}
Note that by (\ref{G}), $\Var G_{n}(t)=t(1-t)$ for any $t\in[0,1]$ fixed. Let us define the standardized empirical process and Brownian bridge
\begin{equation*}
  \widetilde{G}_{n}(t):=\frac{G_{n}(t)}{\sqrt{t(1-t)}},\quad\widetilde{B}(t):=\frac{B(t)}{\sqrt{t(1-t)}},
\end{equation*}
with $\widetilde{G}_{n}(0)=\widetilde{B}(0)=0$. Then, based on (\ref{TECDF}), the Berry--Esseen theorem implies that
\begin{equation*}
  \sup_{x>0}|\PP\{\widetilde{G}_{n}(t)\leq x\}-\PP\{\widetilde{B}(t)\leq x\}|\leq\frac{C_{0}\E|\mathbf{1}(U_{1}\leq t)-t|^{3}}{\sqrt{n(t(1-t))^{3}}}\leq\frac{C_{0}}{\sqrt{nt(1-t)}},
\end{equation*}
where $C_{0}>0$ is the Esseen constant. Thus, by (\ref{varpia}),
\begin{align*}
  \E|\varpi^{1}_{n}-\varpi^{1}_{\infty}|
  &\leq\int^{1}_{0}\E|G_{n}(t)-B(t)|\dd t \\
  &=\int^{1}_{0}\sqrt{t(1-t)}\E|\widetilde{G}_{n}(t)-\widetilde{B}(t)|\dd t \\
  &=\int^{1}_{0}\sqrt{t(1-t)}\int_{\mathbb{R}}|\PP\{\widetilde{G}_{n}(t)\leq x\}-\PP\{\widetilde{B}(t)\leq x\}|\dd x\dd t \\
  &\leq\int^{1}_{0}\sqrt{t(1-t)}\sup_{x\in\mathbb{R}}|\PP\{\widetilde{G}_{n}(t)\leq x\}-\PP\{\widetilde{B}(t)\leq x\}|\dd t \\
  &\leq\frac{C_{0}}{\sqrt{n}},
\end{align*}
where the third line follows from choosing a co-monotonic coupling between $\widetilde{G}_{n}(t)$ and $\widetilde{B}(t)$, for $t\in[0,1]$; see, e.g., \cite[\text{Prop.} 2.1]{C-ART}.

As noted before, the density function $f_{\varpi^{1}_{\infty}}\in\mathcal{C}^{\infty}(\mathbb{R}_{+})$, and based on (\ref{LPDF.1}) and the asymptotic relation (\ref{Kas}), it is clear that $\lim_{x\downarrow0}f_{\varpi^{1}_{\infty}}(x)=0$, and thus $\sup_{x>0}f_{\varpi^{1}_{\infty}}(x)=C_{1}<\infty$. Combining this fact with the last inequality, we have that for any $\epsilon>0$,
\begin{align*}
  \sup_{x>0}|F_{\varpi^{1}_{n}}(x)-F_{\varpi^{1}_{\infty}}(x)|&=\sup_{x>0}|\PP\{\varpi^{1}_{n}\leq x\}-\PP\{\varpi^{1}_{\infty}\leq x\}| \\
  &\leq\sup_{x>0}\PP\{x-\epsilon\leq\varpi^{1}_{\infty}\leq x+\epsilon\}+\PP\{|\varpi^{1}_{n}-\varpi^{1}_{\infty}|>\epsilon\} \\
  &\leq2\epsilon\sup_{x>0}f_{\varpi^{1}_{\infty}}(x)+\frac{\E|\varpi^{1}_{n}-\varpi^{1}_{\infty}|}{\epsilon} \\
  &\leq2\epsilon C_{1}+\frac{C_{0}}{\epsilon\sqrt{n}}.
\end{align*}
By minimizing the last bound over $\epsilon$ we obtain
\begin{equation}\label{lb}
  \sup_{x>0}|F_{\varpi^{1}_{n}}(x)-F_{\varpi^{1}_{\infty}}(x)|\leq2\sqrt{2C_{0}C_{1}}n^{-1/4},
\end{equation}
as desired.
\end{proof}

\medskip

\subsection{Finite-sample computation}\label{sec:2.3}

Analogous to the CvM statistic ($\omega^{2}_{n}=(\varpi^{2}_{n})^{2}$) or the KS statistic ($\varpi^{\infty}_{n}$), a fundamental difficulty in obtaining the exact finite-sample distribution of $\varpi^{1}_{n}$ arises from the dependence structure of the order statistics $U_{(\cdot)}$ that define the empirical distribution function as in (\ref{G}). Generally speaking, since the joint density of the vector $U_{(\cdot)}$ is given by $f_{U_{(\cdot)}}(y)=n!$, for $y\in S_{n}$, computing the exact distribution of any statistic based on the empirical distribution function reduces to integrating the indicator of some convoluted region over $S_{n}$ (\text{cf.} \cite[\text{Sect.} 2]{CF}). In the case of $\varpi^{1}_{n}$, the computational form (\ref{CF1}) allows to explicitly specify such a region, and the resulting distribution function has the general expression
\begin{equation}\label{CDF1}
  F_{\varpi^{1}_{n}}(x):=\PP\{\varpi^{1}_{n}\leq x\}=\E\mathbf{1}(\varpi^{1}_{n}\leq x)=n!\int_{S_{n}}\mathbf{1}(\varpi^{1}_{n}(y)\leq x)\dd y,\quad x\geq0.
\end{equation}

Unfortunately, evaluating the integral in (\ref{CDF1}) explicitly becomes extremely arduous as soon as $n$ exceeds $2$, owing to the complex manner in which the piecewise quadratic structure of (\ref{CF1}) intersects the domain $S_{n}$. Even with the assistance of modern computer algebra systems, we are only able to obtain explicit expressions up to $n=3$, in contrast to the case of $\varpi^{2}_{n}$, where such expressions are available up to $n=7$ (\cite{CF}). To be specific, for $n=1$, we have from (\ref{CF1}) that $\varpi^{1}_{1}=(U_{1}-1/2)^{2}+1/4$, and so
\begin{equation}\label{CDF1n1}
  F_{\varpi^{1}_{1}}(x)=
  \begin{cases}
    0,&\quad x\in\big(0,\frac{1}{4}\big] \\
    \sqrt{4x-1},&\quad x\in\big(\frac{1}{4},\frac{1}{2}\big] \\
    1,&\quad x>1/2.
  \end{cases}
\end{equation}
For $n=2$, the resulting formula is already quite complex. Based on (\ref{CF1}), by writing
\begin{equation*}
  \varpi^{1}_{2}=\sqrt{2}
  \begin{cases}
    \big(U_{(1)}-\frac{1}{4}\big)^{2}-\frac{U_{(2)}}{2}+\frac{7}{16},&\quad U_{(2)}\leq\frac{1}{2} \\
    \big(U_{(1)}-\frac{1}{4}\big)^{2}+\big(U_{(2)}-\frac{3}{4}\big)^{2}+\frac{1}{8},&\quad U_{(1)}\leq\frac{1}{2}<U_{(2)} \\
    \big(U_{(2)}-\frac{3}{4}\big)^{2}+\frac{U_{(1)}}{2}-\frac{1}{16},&\quad U_{(1)}>\frac{1}{2},
  \end{cases}
\end{equation*}
for $0\leq U_{(1)}\leq U_{(2)}\leq1$ and carrying out the integration in (\ref{CDF1}) over $S_{2}$, we deduce that
\begin{equation}\label{CDF1n2}
  F_{\varpi^{1}_{2}}(x)=
  \begin{cases}
    0,&\; x\in\big(0,\frac{1}{4\sqrt{2}}\big] \\
    \pi\big(\sqrt{2}x-\frac{1}{4}\big),&\; x\in\big(\frac{1}{4\sqrt{2}},\frac{3}{8\sqrt{2}}\big] \\
    \pi\big(\sqrt{2}x-\frac{1}{4}\big)+\frac{4}{3}x\sqrt{16\sqrt{2}x-6}-(4\sqrt{2}x-1)\arctan\sqrt{8\sqrt{2}x-3},&\; x\in\big(\frac{3}{8\sqrt{2}},\frac{1}{2\sqrt{2}}\big] \\
    \sqrt{2}x+\frac{1}{8}+\frac{7}{12}\sqrt{2\sqrt{2}x-1}-\frac{2}{3}\sqrt{4\sqrt{2}x-2} \\
    \quad+\frac{1}{3}x\sqrt{16\sqrt{2}x-6-8\sqrt{2\sqrt{2}x-1}}-\frac{1}{8}\sqrt{8\sqrt{2}x-3-4\sqrt{2\sqrt{2}x-1}} \\
    \quad-\frac{1}{6}\sqrt{(2\sqrt{2}x-1)(8\sqrt{2}x-3-4\sqrt{2\sqrt{2}x-1})},&\; x\in\big(\frac{1}{2\sqrt{2}},\frac{5}{8\sqrt{2}}\big] \\
    \sqrt{2}x+\frac{1}{8}+\frac{7}{12}\sqrt{2\sqrt{2}x-1}-\frac{2}{3}\sqrt{4\sqrt{2}x-2} \\
    \quad-\frac{1}{3}x\sqrt{16\sqrt{2}x-6-8\sqrt{2\sqrt{2}x-1}}+\frac{1}{8}\sqrt{8\sqrt{2}x-3-4\sqrt{2\sqrt{2}x-1}} \\
    \quad+\frac{1}{6}\sqrt{(2\sqrt{2}x-1)(8\sqrt{2}x-3-4\sqrt{2\sqrt{2}x-1})},&\; x\in\big(\frac{5}{8\sqrt{2}},\frac{1}{\sqrt{2}}\big] \\
    1,&\; x>\frac{1}{\sqrt{2}}.
  \end{cases}
\end{equation}
Similarly, for $n=3$, (\ref{CF1}) gives that
\begin{equation*}
  \varpi^{1}_{3}=\sqrt{3}
  \begin{cases}
    \big(U_{(1)}-\frac{1}{6}\big)^{2}-\frac{U_{(2)}}{3}-\frac{U_{(3)}}{3}+\frac{17}{36},&\quad U_{(3)}\leq\frac{2}{3},\;U_{(2)}\leq\frac{1}{3} \\
    \big(U_{(1)}-\frac{1}{6}\big)^{2}+\big(U_{(2)}-\frac{1}{2}\big)^{2}-\frac{U_{(3)}}{3}+\frac{1}{3},&\quad U_{(3)}\leq\frac{2}{3},\;U_{(2)}>\frac{1}{3},\;U_{(1)}\leq\frac{1}{3} \\
    \big(U_{(2)}-\frac{1}{2}\big)^{2}+\frac{U_{(1)}}{3}-\frac{U_{(3)}}{3}+\frac{1}{4},&\quad U_{(3)}\leq\frac{2}{3},\;U_{(1)}>\frac{1}{3} \\
    \big(U_{(1)}-\frac{1}{6}\big)^{2}+\big(U_{(3)}-\frac{5}{6}\big)^{2}-\frac{U_{(2)}}{3}+\frac{2}{9},&\quad U_{(3)}>\frac{2}{3},\;U_{(2)}\leq\frac{1}{3} \\
    \big(U_{(1)}-\frac{1}{6}\big)^{2}+\big(U_{(2)}-\frac{1}{2}\big)^{2}+\big(U_{(3)}-\frac{5}{6}\big)^{2}+\frac{1}{12},&\quad U_{(1)}\leq\frac{1}{3}<U_{(2)}\leq\frac{2}{3}<U_{(3)} \\
    \big(U_{(2)}-\frac{1}{2}\big)^{2}+\big(U_{(3)}-\frac{5}{6}\big)^{2}+\frac{U_{(1)}}{3},&\quad U_{(3)}>\frac{2}{3},\;U_{(2)}\leq\frac{2}{3},\;U_{(1)}>\frac{1}{3} \\
    \big(U_{(1)}-\frac{1}{6}\big)^{2}+\big(U_{(3)}-\frac{5}{6}\big)^{2}+\frac{U_{(2)}}{3}-\frac{1}{9},&\quad U_{(2)}>\frac{2}{3},\;U_{(1)}\leq\frac{1}{3},
  \end{cases}
\end{equation*}
for $0\leq U_{(1)}\leq U_{(2)}\leq U_{(3)}\leq1$, for which the distribution function $\PP\{\varpi^{1}_{3}\leq x\}$ is too lengthy for a proper display here.

Because of this complexity, for all $n\geq4$, simulation methods are generally preferred in practice, and instead of attempting to evaluate the simplex integral in (\ref{CDF1}) directly, one repeatedly generates ordered uniform samples and computes $\varpi^{1}_{n}$ via the computational form (\ref{CF1}), which produces accurate approximations to its distribution with minimal computational cost. To give a numerical illustration, we generate $10^{5}$ Monte Carlo simulations of $\varpi^{1}_{n}$ using (\ref{CF1}) with $n=10$ and adopt their empirical distribution function as an approximation of the function $F_{\varpi^{1}_{10}}$. Figure \ref{fig:2} displays this (approximate) distribution function alongside three other distributions: the exact distribution functions $F_{\varpi^{1}_{1}}$ and $F_{\varpi^{1}_{2}}$ obtained from formulae (\ref{CDF1n1}) and (\ref{CDF1n2}), respectively, and the limit null distribution function $F_{\varpi^{1}_{\infty}}$ from (\ref{LCDF.1}), with subscripts emphasizing $\omega^{1}_{n}=\varpi^{1}_{n}$.

\begin{figure}[H]
  \centering
  \includegraphics[scale=0.5]{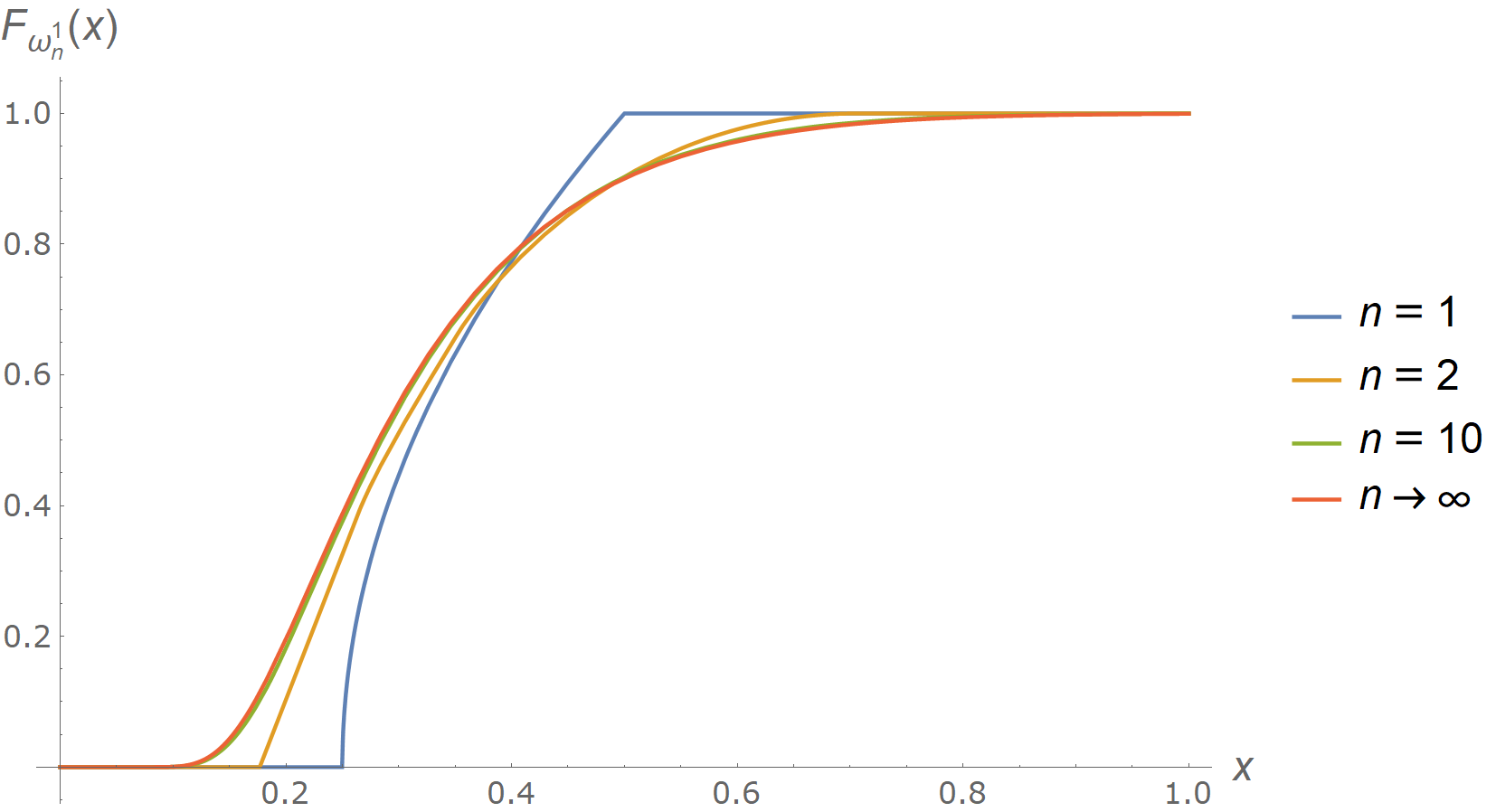}
  \caption{Comparison of null distributions of $\omega^{1}_{n}=\varpi^{1}_{n}$}
  \label{fig:2}
\end{figure}




\medskip

\section{Khmaladze transformation}\label{sec:3}

In this section we apply the Khmaladze transformation to the omega-1 test. As mentioned in Section \ref{sec:1}, the Khmaladze transformation serves the purpose of removing the distributional dependency of the test statistic $\omega^{1}_{n,\theta}$, while not inherently relying on the actual form of the test. If $\theta$ is estimated or hypothesized along with $\mathrm{H}_{0}$, making it a composite hypothesis, or if the requirement that $F_{\theta}$ is absolutely continuous in (\ref{EP}) is dropped, the limit distribution of $\omega^{1}_{n,\theta}$ as $n\to\infty$ is no longer distribution-free. In comparison, its Khmaladze-transformed counterpart, namely $w^{1}_{n,\theta}$ from (\ref{Wp}), remains to have a distribution-free limit.

When implementing the Khmaladze transformation, a major challenge comes in evaluating the integrals in (\ref{Comp}) for the transformed empirical process (\ref{TEP}), with the nonlinearity of the score function $s_{\theta}$ and state-dependent inversion of the Fisher information matrix $\mathcal{I}_{\theta}$ making it a generally difficult task -- even in simple parametric families. By adopting the change of the variables to the general case (\ref{varpi}) as done in (\ref{varpia}), we immediately get an alternative representation for the Khmaladze-transformed test statistic,
\begin{equation*}
  w^{p}_{n,\theta}=\bigg(\int^{1}_{0}|V_{n,\theta}(F^{-1}_{\theta}(t))-\mathcal{K}_{n,\theta}(F^{-1}_{\theta}(t))|^{p}\dd t\bigg)^{1/p},
\end{equation*}
where $\mathcal{K}_{n,\theta}$ is the empirical process compensator and $F^{-1}_{\theta}(q)=\inf\{x\in\mathbb{R}:\;F_{\theta}(x)\geq q\}$, $q\in[0,1]$, is the hypothesized quantile function. Judging from the integrand in (\ref{Comp}), it is implausible to expect a general computational form for (\ref{Wp}) comparable to (\ref{CF1}) in Theorem \ref{thm:1}, and computation of (\ref{Wp}) requires evaluating (\ref{TEP}) with (\ref{Comp}) numerically and should be carried out on a case-by-case basis.

In comparison, as $n\to\infty$, $w^{p}_{n,\theta}$ regains its distribution-free limit, which can be linked to the $\mathrm{L}^{p}$-functional of the standard Brownian motion $W$. More specifically, the following convergence is a consequence of \cite[\text{Sect.} 2]{K1}:
\begin{equation*}
  w^{p}_{n,\theta}\overset{\rm d}{\to}w^{p}_{\infty}:=\bigg(\int^{1}_{0}|W(t)|^{p}\dd t\bigg)^{1/p}.
\end{equation*}
When $p=\infty$ and $p=2$, respectively, the distributions of $w^{\infty}_{\infty}$ and $w^{2}_{\infty}$ are also well-known (see, e.g., \cite[\text{Eq.} 1.15.8.(1)]{BS} and \cite[\text{Eq.} 4.15]{X1}),\footnote{The distribution of $(w^{2}_{\infty})^{2}$ is also known as the Cameron--Martin distribution.} with the following CDFs:
\begin{align*}
  F_{w^{\infty}_{\infty}}(x)&=\frac{1}{2}\sum^{\infty}_{k=-\infty}\bigg(\erfc\frac{(4k-1)x}{\sqrt{2}}-\erfc\frac{(4k+1)x}{\sqrt{2}}\bigg), \\
  F_{w^{2}_{\infty}}(x)&=\sqrt{2}\sum^{\infty}_{k=0}\binom{-1/2}{n}\erfc\frac{4k+1}{2\sqrt{2}x},\quad x>0,
\end{align*}
where $\erfc\equiv\erfc(\cdot)$ denotes Gauss' complementary error function. More importantly, in the case $p=1$, an explicit formula for the distribution function of $w^{1}_{\infty}$ has been recently obtained in \cite[\text{Thm.} 3]{XZ}, though it is somewhat involved. We have
\begin{align}\label{LCDF.1.KT}
  F_{w^{1}_{\infty}}(x)&=\frac{1}{3\sqrt{2\pi}x}\sum^{\infty}_{k=1}\frac{\sqrt{-\alpha'_{k}}}{\Ai(\alpha'_{k})} \bigg(\frac{1}{3}+\frac{(-\alpha'_{k})}{3^{1/3}} \bigg(\frac{1}{3^{1/3}\Gf(2/3)}{\;_{1}\F_{2}}\bigg(\frac{1}{3};\frac{2}{3},\frac{4}{3}\bigg|\frac{(\alpha'_{k})^{3}}{9}\bigg) \nonumber\\
  &\quad+\frac{(-\alpha'_{k})}{2\Gf(1/3)}{\;_{1}\F_{2}}\bigg(\frac{2}{3};\frac{4}{3},\frac{5}{3}\bigg|\frac{(\alpha'_{k})^{3}}{9}\bigg)\bigg) \G^{0,3}_{3,2}\bigg(
  \begin{array}{ccccc}
    5/6 & 7/6 & 3/2 &  &  \\
     &  &  & 1/2 & 1
  \end{array}
  \bigg|\frac{27x^{2}}{2(-\alpha'_{k})^{3}}\bigg),\quad x\geq0,
\end{align}
where ${\;_{1}\F_{2}}\equiv{\;_{1}\F_{2}}(\cdot;\cdot,\cdot|\cdot)$ is a hypergeometric function. Due to the conventional contour integral definition of the Meijer G function, evaluating (\ref{LCDF.1.KT}) is computationally more intense than (\ref{LCDF.1}). For arbitrarily large values of $x$, an asymptotic formula is also available (see \cite{T3} and also \cite[\text{Thm.} 6]{XZ}),
\begin{equation*}
  \PP(w^{1}_{\infty}\leq x)\sim1-\sqrt{\frac{2}{3\pi}}\frac{1}{x}e^{-3x^{2}/2},\quad\text{as }x\rightarrow\infty,
\end{equation*}
also with relative error $O(x^{-2})$. It was also shown (\text{Prop.} 1 \text{ibid.}) that to truncate the series in (\ref{LCDF.1.KT}), one can use the same type of geometric series bound as in Corollary \ref{cor:1}, i.e., if $F_{w^{1}_{\infty}}(x)=\sum^{\infty}_{k=1}\tilde{H}_{k}(x)$,
\begin{equation*}
  \bigg|\sum^{\infty}_{k=K+1}\tilde{H}_{k}(x)\bigg|<\frac{|\tilde{H}_{K}(x)\tilde{H}_{K+1}(x)|}{|\tilde{H}_{K}(x)|-|\tilde{H}_{K+1}(x)|}.
\end{equation*}
Because of this, similar to (\ref{LCDF.1}), truncating the series at $k=6$ ensures an absolute error less than $10^{-5}$ for all $x\in[0,2]$, which would be sufficient for tabulation purposes.

\smallskip

In the following illustration, we adopt the transformed empirical process $\widetilde{V}_{n,\theta}$ under the symmetric Laplace distribution, as has recently been discussed in depth in \cite{RHS}. In this setting, convenient semi-closed forms are available for the compensator (\ref{Comp}), which significantly stabilize numerical integration by avoiding explicit matrix inversion in the integrand; we refer to \cite[\text{Sect.} 2]{RHS} for more details. Hence, following the procedures therein, in Figure \ref{fig:3} we compare the empirical distribution function of the resulting (transformed) omega-1 test statistic $w^{1}_{n,\theta}$ for $n=100$ and $\theta=\hat{\theta}_{n}$ (maximum likelihood estimators), when the true distribution is a standard Laplace distribution (with scale parameter $1$), based on $10^{4}$ Monte Carlo simulations, against the limit distribution in (\ref{LCDF.1.KT}). The former is taken as an approximation of the distribution function of $w^{1}_{n,\theta}$ under $\mathrm{H}_{0}$, which is simply denoted as $F_{w^{1}_{n,\theta}}$.

\begin{figure}[H]
  \centering
  \includegraphics[scale=0.5]{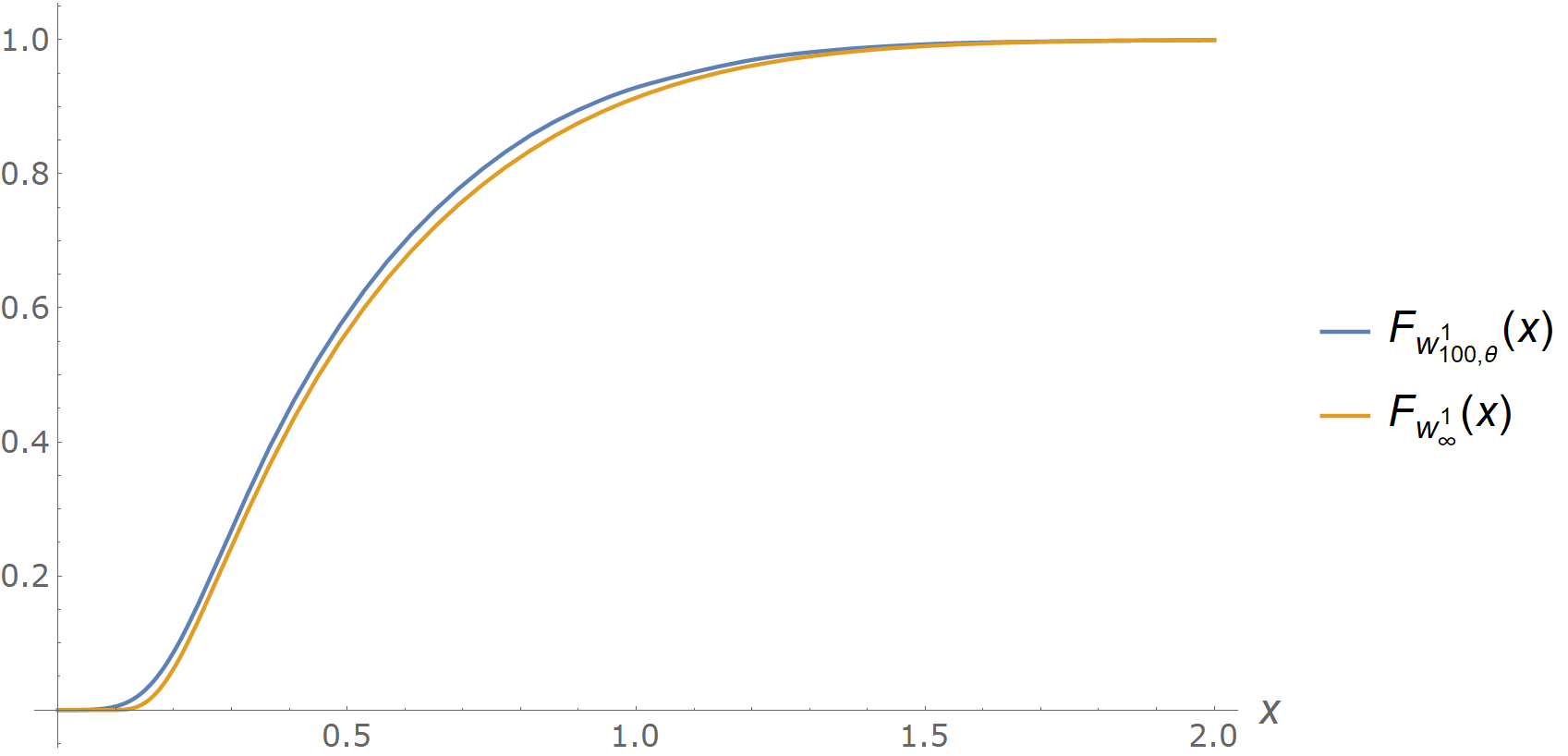}
  \caption{Comparison of (null) distributions of $w^{1}_{n}$ for $n=100$ under Laplace distribution ($\theta=\hat{\theta}_{n}$) and $w^{1}_{\infty}$}
  \label{fig:3}
\end{figure}

This illustration mainly aims at highlighting the even-faster convergence to the asymptotic limit, which feature is consistent with the behavior of the untransformed statistic depicted in Figure \ref{fig:2}. What is observed is also directly comparable to \cite[\text{Figs.} 1 \& 2]{RHS}, which together provide empirical evidence of accelerated convergence as $p$ decreases (with $1$ being the smallest possible value). As a consequence, the omega-1 statistic can seamlessly integrate with the Khmaladze transformation, whenever (\ref{Comp}) can be efficiently computed, by which the limit distribution in (\ref{LCDF.1.KT}) can serve as a universal reference distribution tabulable for moderate to large sample sizes.


\bigskip

\begin{appendices}

\section{An alternative computational form for omega-1 test statistic}\label{A}

\renewcommand{\theequation}{A.\arabic{equation}}

Following (\ref{G}) and (\ref{TECDF}), for every $t\in[0,1]$, consider the set $A_{n}(t):=\{G_{n}(t)\geq0\}$, and observe that under $\mathrm{H}_{0}$, it holds that up to null sets,
\begin{equation*}
  A_{n}(t)=\bigcup^{n}_{j=1}\bigg\{t\in\bigg[\frac{j-1}{n},a_{j}(U_{(j)})\bigg)\bigg\},\quad t\in[0,1],
\end{equation*}
with a restriction function $a_{j}(x)=\max\{\min\{x,j/n\},(j-1)/n\}$, $x\in[0,1]$. Thus, with $(\cdot)^{+}$ denoting the positive part,
\begin{align}\label{Gp}
  \int^{1}_{0}G^{+}_{n}(t)\dd t&=\sqrt{n}\sum^{n}_{j=1}\int^{a_{j}(U_{(j)})}_{(j-1)/n}(\breve{F}_{n}(t)-t)\dd t \nonumber\\
  &=\sqrt{n}\sum^{n}_{j=1}\bigg(\frac{1}{2}\bigg(a^{2}_{j}(U_{j})-\frac{(j-1)^{2}}{n^{2}}\bigg) \nonumber\\
  &\qquad-\frac{1}{n}\sum^{n}_{m=1}\bigg((a_{j}(U_{(j)})-U_{(m)})^{+}-\bigg(\frac{j-1}{n}-U_{(m)}\bigg)^{+}\bigg)\bigg).
\end{align}
Similarly, note that up to null sets,
\begin{equation*}
  [0,1]\setminus A_{n}(t)=\bigcup^{n}_{j=1}\bigg\{a_{j}(U_{(j)})\leq t<\frac{j}{n}\bigg\},\quad t\in[0,1],
\end{equation*}
so we have
\begin{align}\label{Gm}
  \int^{1}_{0}G^{-}_{n}(t)\dd t&=\sqrt{N}\sum^{n}_{j=1}\int^{j/n}_{a_{j}(U_{(j)})}(\breve{F}_{n}(t)-t)\dd t \nonumber\\
  &=\sqrt{n}\sum^{n}_{j=1}\bigg(\frac{1}{2}\bigg(a^{2}_{j}(U_{(j)})-\frac{j^{2}}{n^{2}}\bigg) -\frac{1}{n}\sum^{n}_{m=1}\bigg((a_{j}(U_{(j)})-U_{(m)})^{+}-\bigg(\frac{j}{n}-U_{(m)}\bigg)^{+}\bigg)\bigg).
\end{align}
Adding up (\ref{Gp}) and (\ref{Gm}) we then obtain
\begin{align*}
  \varpi^{1}_{n}&=\sqrt{n}\sum^{n}_{j=1}\bigg(a^{2}_{j}(U_{(j)})-\frac{1}{n}\sum^{n}_{m=1}\bigg(2(a_{j}(U_{(j)})-U_{(m)})^{+} -\bigg(\frac{j-1}{n}-U_{(m)}\bigg)^{+}-\bigg(\frac{j}{n}-U_{(m)}\bigg)^{+}\bigg)\bigg) \\
  &\quad-\frac{2n^{2}+1}{6\sqrt{n}},
\end{align*}
which gives another computational form alternative to (\ref{CF1}).

\medskip

\section{A computational form for Wasserstein 1-distance}\label{B}

As mentioned in Subsection \ref{sec:2.1}, the statistic $\omega^{1}_{n}$ is equivalent to the Wasserstein 1-distance between the empirical distribution of an \text{i.i.d.} uniform sample (as in (\ref{TECDF})) and the standard uniform distribution. In the following we present a computational form for the Wasserstein 1-distance between a general empirical distribution and theoretical (reference) distribution.

Consider the setting of (\ref{EP}). By following the proof of Theorem \ref{thm:1} we obtain
\begin{align*}
  \mathcal{W}_{1}(F_{n},F_{\theta})&:=\int_{\mathbb{R}}|F_{n}(t)-F_{\theta}(t)|\dd t \\
  &=\int^{1}_{0}|F^{-1}_{\theta}(u)-F^{-1}_{n}(u)|\dd t \\
  &=\sum^{n}_{j=1}\int^{j/n}_{(j-1)/n}|F^{-1}_{\theta}(u)-X_{(j)}|\dd u \\
  &=\sum^{n}_{j=1}\int^{F^{-1}_{\theta}(j/n)}_{F^{-1}_{\theta}((j-1)/n)}|t-X_{(j)}|\dd F_{\theta}(t),
\end{align*}
where $F_{n}$ is the empirical CDF of the \text{i.i.d.} sample $X$ and $F^{-1}_{\theta}(q):=\inf\{t\in[0,1]:\;F_{\theta}(t)\geq q\}$, $q\in[0,1]$, is the theoretical quantile function. Thus, if $F_{\theta}$ is also continuous, we have the computational form
\begin{equation}\label{W1}
  \mathcal{W}_{1}(F_{n},F_{\theta})=\sum^{n}_{j=1}G_{\theta}\bigg(F^{-1}_{\theta}\bigg(\frac{j-1}{n}\bigg), F^{-1}_{\theta}\bigg(\frac{j}{n}\bigg);X_{(j)}\bigg),
\end{equation}
where for $0\leq a\leq b\leq1$ and $x\in\mathbb{R}$,
\begin{equation}\label{GF}
  G_{\theta}(a,b;x):=
  \begin{cases}
    \displaystyle g_{\theta}(b,x)-g_{\theta}(a,x),&\quad \displaystyle x\leq a  \\
    \displaystyle g_{\theta}(a,x)-2g_{\theta}(x,x)+g_{\theta}(b,x),&\quad \displaystyle a<x\leq b \\
    \displaystyle g_{\theta}(a,x)-g_{\theta}(b,x),&\quad \displaystyle x>b,
  \end{cases}
  \quad g_{\theta}(t,x)=\int(t-x)\dd F_{\theta}(t),
\end{equation}
which can be taken as a functional extension of (\ref{CF1c}), where in particular $F_{\theta}(t)=t$ and $X=U$.

Therefore, the computational form (\ref{W1}) admits a closed-form expression whenever both the antiderivative $g_{\theta}$ and the quantile function $F^{-1}_{\theta}$ in (\ref{GF}) are available in closed form, which property is satisfied by most standard continuous distributions -- including the generalized normal distribution family (see \cite{N}). Indeed, consider without loss of generality the centered case with location parameter $\mu=0$, along with general scale parameter $\alpha>0$ and shape parameter $\beta>0$. Then, we have that up to a constant shift,
\begin{align*}
  g_{\theta}(t,x)&=\int^{t}_{-\infty}\frac{\beta(y-x)}{2\alpha\Gf(1/\beta)}e^{-(|y-\mu|/\alpha)^{\beta}}\dd y \\
  &=\frac{x\sgn(t)\Gf(1/\beta,(|t|/\alpha)^{\beta})-\alpha\Gf(2/\beta,(|t|/\alpha)^{\beta})}{2\Gamma(1/\beta)}-x\mathbf{1}(t\geq0),\quad t\in[0,1],\;x\in\mathbb{R},
\end{align*}
where $\Gf(\cdot,\cdot)$ denotes the upper incomplete gamma function, while
\begin{equation*}
  F^{-1}_{\theta}(q)=\alpha\mathrm{Q}^{1/\beta}\big(\tfrac{1}{\beta},1-2\big|q-\tfrac{1}{2}\big|\big),\quad q\in[0,1],
\end{equation*}
with the inverse regularized (upper) incomplete gamma function $\mathrm{Q}\equiv\mathrm{Q}(\cdot,\cdot)$. Plugging these back in (\ref{W1}) yields a closed-form expression for the (1D) Wasserstein 1-distance between an empirical distribution and a generalized normal distribution.

Generally speaking, evaluating (\ref{W1}) numerically is also much more stable than direct integration involving the CDFs, as the latter features an integrand with numerous jumps due to the empirical CDF. In this consideration, the formula (\ref{W1}) along with (\ref{GF}) provides immediate computational benefits for related applications centered around the Wasserstein 1-distance -- notably generative adversarial networks and domain adaptation.

\medskip

\section{Omega-1 table}\label{C}

For reference, the below table provides critical values of the omega-1 test -- specifically, values of $x$ such that $\PP\{\omega^{1}_{n}\leq x\}=\alpha$ (significance level) under $\mathrm{H}_{0}$, for sample sizes up to $n=20$. The third to fifth rows are obtained by utilizing the explicit formulas mentioned in Subsection \ref{sec:2.3}, and the last row by implementing (\ref{LCDF.1}) with the series truncated at $k=4$. The rows in between are constructed from $10^{5}$ Monte Carlo simulations via the computational form (\ref{CF1}), following the same approach used for Figure \ref{fig:2}. In view of standard KS and CvM table conventions, the reported values are rounded to three decimal places.

\begin{table}[H]\small
  \centering
  \caption{Critical values of omega-1 test}
  \label{tab:1}
  \begin{tabular}{c|c|c|c|c|c}
    \Xhline{2\arrayrulewidth}
    \multirow{2}{*}{Sample size ($n$)} & \multicolumn{5}{c}{significance level ($\alpha$)} \\ \cline{2-6}
    & 20\% & 15\% & 10\% & 5\% & 1\% \\ \hline
    1 & 0.26 & 0.256 & 0.252 & 0.251 & 0.25 \\
    2 & 0.222 & 0.211 & 0.199 & 0.188 & 0.179 \\
    3 & 0.215 & 0.201 & 0.187 & 0.171 & 0.153 \\
    4 & 0.211 & 0.198 & 0.184 & 0.166 & 0.144 \\
    5 & 0.209 & 0.196 & 0.181 & 0.164 & 0.139 \\
    6 & 0.208 & 0.195 & 0.18 & 0.162 & 0.136 \\
    7 & 0.207 & 0.194 & 0.179 & 0.16 & 0.134 \\
    8 & 0.207 & 0.193 & 0.178 & 0.16 & 0.133 \\
    9 & 0.206 & 0.192 & 0.177 & 0.158 & 0.132 \\
    10 & 0.205 & 0.192 & 0.177 & 0.158 & 0.132 \\
    11 & 0.204 & 0.191 & 0.176 & 0.157 & 0.131 \\
    12 & 0.205 & 0.191 & 0.176 & 0.158 & 0.131 \\
    13 & 0.204 & 0.191 & 0.176 & 0.157 & 0.13 \\
    14 & 0.204 & 0.19 & 0.176 & 0.157 & 0.13 \\
    15 & 0.204 & 0.19 & 0.175 & 0.157 & 0.129 \\
    16 & 0.203 & 0.19 & 0.175 & 0.156 & 0.129 \\
    18 & 0.203 & 0.19 & 0.175 & 0.156 & 0.128 \\
    20 & 0.203 & 0.189 & 0.174 & 0.155 & 0.128 \\ \hline
    limit & 0.201 & 0.187 & 0.172 & 0.153 & 0.126 \\
    \Xhline{2\arrayrulewidth}
  \end{tabular}
\end{table}

\end{appendices}

\end{document}